\documentclass[12pt,reqno]{amsart}

\newcommand{\R}{\mathbb{R}}
\newcommand{\Hk}{H^0(M,L^k)}
\newcommand{\Z}{\mathbb{Z}}
\newcommand{\C}{\mathbb{C}}
\newcommand{\T}{\mathbb{T}^{2m}}
\newcommand{\kahler}{K\"ahler }

\newcommand{\hess}{\nabla^2}

\newcommand{\CP}{\mathbb{CP}}

\usepackage{color}
\usepackage{amsthm}
\usepackage{amsmath}
\usepackage{amstext}
\usepackage{amssymb}
\usepackage{amsfonts}
\usepackage{latexsym}
\usepackage{amscd}

\newcommand{\bcal}{\mathcal{B}}

\newcommand{\hcal}{\mathcal{H}}

\newcommand{\ocal}{\mathcal{O}}

\newtheorem{theo}{{\sc Theorem}}[section]
\newtheorem{cor}[theo]{{\sc Corollary}}

\newtheorem{lemma}[theo]{{\sc Lemma}}
\newtheorem{prop}[theo]{{\sc Proposition}}

\begin{document}
\title []
{Bergman metrics and geodesics in the space of
K\"{a}hler metrics on principally polarized Abelian varieties}
\author[Renjie Feng]{Renjie Feng}
\address{Department of Mathematics, Northwestern University, USA}
\email{renjie@math.northwestern.edu}
\date{\today}
\maketitle
\begin{abstract}  It's well-known in \kahler geometry that the
infinite dimensional symmetric space  $\hcal$ of smooth \kahler metrics
in a fixed \kahler  class on a polarized \kahler
manifold is well approximated by finite dimensional
submanifolds $\bcal_k \subset \hcal$ of Bergman metrics of height
$k$. Then it's natural to ask whether  geodesics in $\hcal$ can be
approximated by Bergman geodesics in  $\bcal_k$. For any polarized \kahler manifold, the approximation is in the $C^0$ topology. While Song-Zelditch proved the $C^2$ convergence for the torus-invariant metrics over toric varieties. In this article, we show that some $C^{\infty}$ approximation exists as well as a complete asymptotic expansion for principally polarized Abelian varieties. We also get a $C^\infty$ complete asymptotic expansion for harmonic maps into $\bcal_k$ which generalizes the work of Rubinstein-Zelditch on toric varieties.

 \end{abstract}
\section{Introduction}
Let $(M,\omega)$ be an $m$-dimensional polarized \kahler manifold. Then the space $\hcal$ of smooth \kahler metrics in a fixed \kahler class will be an infinite dimensional Riemannian manifold under the natural $L^2$ metric. At the level of individual metrics $\omega \in \hcal$, there exists a well-developed approximation theory \cite{T,Z2}: Given $\omega$, one can define a canonical sequence of Bergman metrics $\omega_k \in \bcal_k$ which approximates $\omega$ in the $C^\infty$ topology. The approximation theory is
based on microlocal analysis in the complex domain, specifically Bergman kernel asymptotics
on and off the diagonal. Our principal aim
is to study the approximation of certain
global aspects of the geometry, such as the approximation of the harmonic maps or geodesics in  $\hcal$ by the corresponding objects in $\bcal_k$.

The geodesic equation
for the \kahler potentials $\phi_t$ of $\omega_t$ is a complex homogeneous Monge-Amp\`{e}re equation \cite{D,S}. Concerning the solution of this Dirichlet problem, we have the following regularity theorem: $\phi_t\in C^{1,\alpha}([0, T] \times M)$ for all $\alpha< 1$ if the endpoint metrics are smooth \cite{C}. It is therefore natural to study the
approximation of Monge-Amp\`{e}re geodesics $\phi_t$ in $\hcal$ by the much simpler geodesics $\phi_k(t, z)$ in $\bcal_k$, which are defined by one parameter subgroups of $GL(d_k + 1)$. The problem of
approximating geodesic segments in $\hcal $ between two smooth endpoints by  geodesic segments in $\bcal_k$
was raised by Arezzo-Tian, Donaldson  and Phong-Sturm \cite{AT, D, PS}). Phong-Sturm  proved that $\phi_k(t, z)\rightarrow \phi_t$ in a weak
$C^0$ sense on $[0, 1]\times M$; a $C^0$ result with a remainder estimate was later proved by
Berndtsson \cite{B}.

To understand the approximation of $\hcal$-geodesics by $\bcal_k$-geodesics better, e.g.,
the rate of the approximation, we can test some special varieties and expect a better result. For example, in the toric varieties case, when one restricts to
torus-invariant metrics, the geodesic equation becomes the real homogeneous Monge-Amp\`{e}re equation and thus can be linearized by the
Legendre transform \cite{S}. Thus the geodesic will be smooth if the endpoints are two smooth metrics. For such geodesics, Song-Zelditch proved a profound $C^2$ convergence in space-time
 derivatives with remainder estimates. In a subsequent paper by Rubinstein-Zelditch \cite{RZ}, it was proved that the harmonic map equation can be
linearized and thus can be solved and that harmonic maps into $\hcal$ are also the $C^2$  limit of the cooresponding ones into $\bcal_k$.

 Our motivation in this article is to test the convergence of geodesics and more general harmonic maps over the principally polarized Abelian varieties by applying the method developed in \cite{RZ, SoZ}. Our main result is that $\phi_k(t,z)\rightarrow \phi_t(z)$ in the $C^\infty$ topology in this Abelian case. Moreover, $\phi_k(t,z)$ has a complete asymptotic expansion in $k$ with the leading term $\phi_t(z)$ and the second term $\log (k^m R_\infty)$ where $R_\infty$ is the ratio of the norming constants (\ref{uyee}). We also test the convergence of the harmonic maps into $\hcal_0^\Gamma$ of $(S^1)^m$-invariant metrics by the corresponding ones into $\bcal_k$ and the convergence is still in the $C^\infty$ topology.



\section{Background}
\subsection{Geodesics in $\hcal$ and $\bcal_k$}Let $M$ be an $m$-dimensional compact \kahler manifold, $L\rightarrow M$ an ample holomorphic
line bundle. Let $h$ be a
smooth hermitian metric on $L$, then $h^k$ will be the
induced metric on $L^k$. The curvature of $h$ is the $(1,1)$-form on $M$ defined locally by the
formula $R(h) =-\frac{\sqrt{-1}}{2} \partial \bar \partial \log |s(z)|^2_h$, where $s(z)$ is a local, nowhere vanishing holomorphic
section \cite{GH}.  If we fix a hermitian metric $h_0$ and let $\omega_0 = R(h_0)$, then we define $\hcal$ as the space of \kahler metrics in the fixed class of $[\omega_0]$:
 \begin{equation} \label{HCALDEF} \hcal\ = \
\{\phi\in C^{\infty} (M) : \omega_\phi\ = \ \omega_0+ \frac{\sqrt{-1}}{2}\partial \bar \partial \phi>0\
\},
\end{equation}where $\phi$ is identified with $h = h_0e^{-\phi}$ so that $R(h)=\omega_{\phi}$.
If we define
the metric $g_{\hcal}$ on $\hcal$ as
\begin{equation} \label{metric} ||\psi||^2_{g_{\hcal}, \phi}\ = \ \int_M |\psi|^2\
\omega_{\phi}^m,\ \, \;\; {\rm ~ where~} \phi \in \hcal
{\rm~ and ~} \psi \in T_{\phi} \hcal \simeq C^{\infty}(M).
\end{equation}
Then formally $(\hcal, g_{\hcal})$ is an infinite
dimensional non-positively curved  symmetric  Riemannian manifold \cite{D, M, S}. Furthermore, the geodesics of $\hcal$ in this metric are the paths $\phi_t$ which satisfy the partial
differential equation: \begin{equation}\label{eeew} \ddot{\phi}-|\partial \dot{\phi}|^2_{\omega_{\phi}}=0.\end{equation}

The space $\hcal$ contains a family of finite-dimensional non-positively curved symmetric
spaces $\bcal_k$ which are defined as follows: Let
  $H^0(M, L^k)$ be the space of  holomorphic
sections of $L^k \to M$ and let $d_k + 1 =
\dim H^0(M, L^k)$.
 For large $k$ and for $\underline{s} = (s_0, ...., s_{d_k})$ an ordered
basis of  $H^0(M, L^k)$, let
\begin{equation}  \iota_{\underline{s}}: M \rightarrow \mathbb{CP}^{d_k},\;\;z \rightarrow [s_0(z),
\dots, s_{d_k}(z)] \end{equation}
be the Kodaira embedding. Then we have a canonical
isomorphism $L^k =  \iota_{\underline{s}}^* O(1)$. We then define a Bergman metric of height $k$ to be a metric of the form:
 \begin{equation} \label{FSDEFa}  FS_k(\underline{s}):= (\iota_{\underline{s}}^*
h_{FS})^{1/k} = \frac{h_0}{\left( \sum_{j = 0}^{d_k}
|s_j(z)|^2_{h_0^k} \right)^{1/k}}, \end{equation} where $h_{FS}$
is the Fubini-Study Hermitian metric on $\ocal(1) \to \CP^{d_k}$. Note that the right side of (\ref{FSDEFa}) is
independent of the choice of $h_0$.
We define the space of Bergman metrics as:
\begin{equation} \label{FSDEFaww}\bcal_k = \ \{FS_k(\underline{s}):  \underline{s} \hbox{\ a basis
of  $H^0(M, L^k) $\}. }\  \
\end{equation}
Then $\bcal_k=GL(d_k+1)/U(d_k+1)$ is a finite-dimensional
negatively curved symmetric space.  It's proved in \cite{T,Z2} that the union ${\mathcal B} = \bigcup_{k=1}^{\infty} \bcal_k$ is dense in $\hcal$ in the $C^{\infty}$ topology : If $h \in \hcal$, then there exists $h(k) \in \bcal_k$ such that $h(k) \rightarrow h$ in $C^{\infty}$ topology.

In fact, there is a canonical choice of the approximating sequence $h(k)$ \cite{T} which is used throughout the article. The hermitian metric $h$ on $L$ induces a natural inner product $Hilb_k(h)$ on $\Hk$ defined by:
\begin{equation}\label{dsldd} \langle s_1, s_2\rangle_{h^k}=\int_M (s_1(z),s_2(z))_{h^k}\frac{\omega_h^m}{m!} ,\; \;\mbox{where}\; \omega_h=R(h),\end{equation} for any $s_1,s_2\in \Hk$. In particular, the norm square of the holomorphic section is:
\begin{equation}\label{dsldddd} \|s\|^2_{h^k}=\int_M |s|^2_{h^k}\frac{\omega_h^m}{m!} \end{equation}
Now choose $\underline{s}(k)$ as an orthonormal basis of $\Hk$ with respect to the inner product $Hilb_k(h)$, then we have the following $C^{\infty}$ asymptotics for the Bergman kernel as $k\rightarrow \infty$ \cite{Z2} (see also \cite{BBS, BS}):\begin{equation} \label{gbvgg}\sum_{j = 0}^{d_k}
|s_j(z)|^2_{h^k}=k^m+a_1(z)k^{m-1}+\cdots ,\end{equation}where $a_1(z)$ is the scalar curvature of $h$. Now let $\underline{\hat{s}}(k)=k^{-\frac{m}{2}}\underline{s}(k)$. Then the Bergman metric $h(k)=FS_k\circ Hilb_k(h):=FS_k(\underline{\hat{s}}(k))$ will be an approximating sequence of $h$; to be more precise, (\ref{FSDEFa}) and (\ref{gbvgg}) imply that for each $r>0$,\begin{equation} \left \|\frac{h(k)}{h}-1\right\|=O(\frac{1}{k^2})\,\,,\left \|\omega(k)-\omega\right\|=O(\frac{1}{k^2})\,\,,\left \|\phi(k)-\phi\right\|=O(\frac{1}{k^2}),\end{equation} where the norms are taken with respect to $C^r(\omega_0)$. Here, as before, $\omega=R(h)$, $\omega(k)=R(k)$, $h=h_0e^{-\phi}$, $h(k)=h_0e^{-\phi(k)}$.

Now we can
compare geodesics in $\hcal$ and Bergman geodesics in $\bcal_k$. Let $h_0,h_1\in\hcal$. Then there will be a unique
$C^{1,\alpha}$ Monge-Amp\`{e}re geodesic $h_t=h_0e^{-\phi_t(z)}$: $[0,1]$$\rightarrow \hcal$ joining
$h_0$ to $h_1$ for all $\alpha\in (0,1)$ \cite{C}. Assume $h_0(k)=FS_k(\hat {\underline{s}}^{(0)}(k))$ and $h_1(k)=FS_k(\hat {\underline{s}}^{(1)}(k))$ are two sequence in $\bcal_k$ obtained by the canonical construction approximating $h_0$ and $h_1$. Then the geodesic joining $h_0(k)$ and $h_1(k)$ in the space $\bcal_k=GL(d_k+1)/U(d_k+1)$ is constructed in \cite{PS} as follows :
 Let $\sigma_k \in GL(d_k+1)$ be the change of basis
matrix defined by
$\sigma_k \cdot \hat{\underline{s}}^{(0)}(k)=\hat {\underline{s}}^{(1)}(k)$. Without loss of
generality, we may assume that
$\sigma_k$ is diagonal with entries
$e^{\lambda_0},...,e^{\lambda_{d_k}}$ for some
$\lambda_j\in\R$. Let $\hat {\underline{s}}^{(t)}(k)=\sigma_k^t\cdot\hat {\underline{s}}^{(0)}(k)$
where $\sigma_k^t$ is diagonal with entries
$e^{\lambda_jt}$. Define
\begin{equation}
\label{phdfcid}
h_k(t,z)=FS_k(\hat {\underline{s}}^{(t)}(k))=h_0e^{-\phi_k(t,z)}.
\end{equation}
Then $h_t(k,z)$ is the smooth geodesic in
$GL(d_k+1)/U(d_k+1)$ joining $h_0(k)$ to
$h_1(k)$. Explicitly, use identity (\ref{FSDEFa}) again, we have:
\begin{equation}
\label{phid}
   \phi_k(t,z)\ = \frac{1}{k}
\log \left(\sum_{j=0}^{d_k}e^{2\lambda_jt}|\hat s_j^{(0)}(k)|^2_{h_0^k}
\right).
\end{equation}
Then the main result of Phong-Sturm \cite{PS} is that the Monge-Amp\`{e}re geodesic $\phi _t(z)$ is approximated by Bergman geodesic $\phi_k(t,z)$
in a weak $C^0$ sense on $[0, 1]\times M$; a $C^0$ result with a remainder estimate was later proved by
Berndtsson \cite{B}.

For special varieties, one expects better results. The first evidence is in \cite{SoZ}:  Song-Zelditch proved the convergence of $\phi_k(t, z) \rightarrow \phi_t(z)$ is much
stronger for toric hermitian metrics on the torus-invariant line bundle over the smooth toric \kahler manifold. To be more precise, define the space
of toric Hermitian metrics:
 \begin{equation} \hcal(\mathbb{T}^m) = \{\phi \in \hcal: (e^{i \theta})^* \phi = \phi, \;\; {\rm ~for~all~} e^{i \theta}
  \in \mathbb{T}^m\} \end{equation}
Then for the smooth geodesic in $\hcal(\mathbb{T}^m) $ with the endpoints $h_0$ and $h_1 \in \hcal(\mathbb{T}^m)$, they proved:
\begin{equation}\lim_{k\rightarrow \infty}\phi_k(t,z)=\phi(t,z) \;\; \mbox{in} \; C^2([0,1]\times M)\end{equation}
And they also obtained the rate of the convergence and the remainder estimates. In fact, their method can be applied to the principally polaried Abelian varieties. In our article, we consider the Abelian case and prove the existence of $C^\infty$ convergence, moreover, we can expand $\phi_k(t,z)$ in $k$ completely with the leading term $\phi_t$.
\subsection{$\Gamma$-invariant space $\hcal_0^\Gamma$}

 Throughout the article, we will use the following notation: denote $\Gamma=(S^1)^m\cong (\R/\Z)^m$, the isomorphism is given by $e^{2\pi i \theta}\rightarrow \theta \mod \Z^m$; thus we can identify a periodic function on $\R^m$ with period 1 in each variable with a function defined on $\Gamma$; denote $y^2=y_1^2+\cdots+y^2_m$ and $x \cdot y=x_1y_1+\cdots+x_my_m$ for $x, y\in \R^m$.

 By performing affine transformation, it suffices to consider the principally polarized Abelian variety $M=\C^m/\Lambda$,
 where $\Lambda=\Z^m+i\Z^m$. We will prove our main result for this model case first and in section \ref{general}, we will sketch how to extend our argument to the general lattice.

  Now for $M=\C^m/\Lambda$, where $\Lambda=\Z^m+i\Z^m$, we can write each point in $M$ as $z=x+iy$, where $x, y \in \R^m$ and they can be considered as the period coordinate in $M$. There is a natural action on $M$: the group $\Gamma=(S^1)^m$ acts on $M$ via translations in the Langrangian subspace $\R^m \subset \C^m$, i.e., the translation of $x$ in the universal covering space.

 Let $L \rightarrow M$ be a principal polarization of $M$; then there exists a hermitian metric defined on $L$ \cite{GH}:$$h=e^{-2\pi y^2}$$ The curvature of $h$ is given by $R(h)=\frac{ \sqrt{-1}}{2}\pi \sum_{\alpha=1}^m dz_{\alpha} \wedge d\bar z_{\alpha}$ which is in the class $[\pi c_1(L)]$. Now fix $\omega_0=R(h)$ a flat metric on $M$ with associated \kahler potential $2\pi y^2$, denote $\hcal_0^{\Gamma}$ as the space of $\Gamma$-invariant \kahler metrics in the fixed class $[\omega_0]$, then:
    $$\hcal_0^{\Gamma}\ = \
\{\psi\in C^{\infty} _{\Gamma}(M) : \omega_\psi\ = \ \omega_0+ \frac{\sqrt{-1}}{2}\partial \bar \partial \psi>0\}.$$
Note that a smooth function $\psi(x,y)$ defined on $M$ invariant under the $\Gamma$ action should be independent of $x$ variable; thus in fact it induces a smooth function on $M/\Gamma$, i.e., $\psi$ can be considered as a smooth and periodic function on the universal covering space $y\in \R^m$.

All hermitian metrics $h$ on $L$ such that $R(h)=\omega_{\psi}\in \hcal_0^{\Gamma}$ are of the form: \begin{equation}h=e^{-2\pi y^2-4\pi \psi(y)}.\end{equation} In section \ref{jhnmb}, we will see such $h$ is a well defined hermitian metric on $L$. And the corresponding \kahler potential is:  \begin{equation}\label{dhgdg}\varphi(y)=2\pi y^2+ 4\pi \psi(y),\end{equation} where $\psi(y)$ is a smooth and periodic function with period $1$ and $\hess \varphi(y)>0$.

The following fact about the space $\hcal_0^{\Gamma}$ is crucial \cite{D,S}: Given any $\varphi_0$ and $\varphi_1$ $\in \hcal_0^{\Gamma}$,
 we can join them by a smooth geodesic $\varphi_t \in \hcal_0^{\Gamma}$.
 Thus throughout the article, we will consider the geodesic in the form $\varphi_t(y)=2\pi y^2+4 \pi \psi_t(y)$.

In section \ref{general}, we show how to get our main results for the case of a general lattice.

\section{Main results}
\subsection{Complete asymptotics of geodesics}
Our main task in this article is to prove the following theorem:
\begin{theo} \label{ghjiuyum}Let $M$ be a principally polarized Abelian variety and let $L \rightarrow M $ be a principal polarization of $ M$. Given $h_0$ and $h_1$ in $\hcal_0^{\Gamma}$ of the space of $\Gamma$-invariant \kahler metrics, let $h_t \in \hcal_0^{\Gamma}$ be the smooth geodesic between them. Let $h_k(t)$ be the Bergman geodesic between $h_0(k)$ and $h_1(k)$ in $\bcal_k$. Let $h_k(t)=e^{-\phi_k(t,z)}h_0$ and $h_t=e^{-\phi_t(z)}h_0$, then, $$\lim_{k\rightarrow \infty}\phi_k(t,z)=\phi_t(z)$$  in the $C^{\infty}([0,1] \times M)$ topology. Moreover, we have the following $C^{\infty}$ complete asymptotics: \begin{equation}\phi_k(t,z)=\phi_t(z)+mk^{-1}\log k+k^{-1}a_1(t,\mu_t)+k^{-2}a_2(t,\mu_t)+\cdots\end{equation} for $k$ large enough, where $\mu_t(y)=\nabla \varphi_t(y)$ where $y$ is defined in (\ref{jhgui}) and each $a_n$ is a smooth function of $\mu_t$ and $t$. In particular, $a_1=\log R_{\infty}$ where $R_{\infty}$ is defined by (\ref{uyee}). \end{theo}
We now sketch the proof of our main result for the model case: define the inner product on $\Hk$ induced by $h_t^k$ in the sense of (\ref{dsldd}), then in Proposition \ref{dfghhg} we first prove that: for any fixed $t$,  the following theta functions of level $k$:
  $$\label {a}\theta_j(z)=\sum_{n \in \Z^m}e^{-\pi \frac{(j+kn)^2}{k} +2\pi i(j+kn)\cdot z} \,\,, j\in (\Z/k\Z)^m$$form an orthogonal basis with respect to this inner product, in particular, $\dim \Hk=k^m$. Therefore, we can choose the orthonormal basis $\underline{s}^{(t)}(k)$  as $\theta_j$ normalized by $\|\theta_j\|_{h_t^k}$. Hence, if $\sigma_k \in GL(k^m)$ such that $\sigma_k \cdot \underline{\hat{s}}^{(0)}(k)=\underline{\hat{s}}^{(1)}(k)$, then $\sigma_k$ can be chosen to be diagonal with entries $e^{\lambda_j}=\|\theta_j\|_{h^k_0}/ \|\theta_j\|_{h^k_1}$. Hence, the equation (\ref{phid}) of the Bergman geodesic becomes:
 \begin{equation}
\label{phdsdid}
   \phi_k(t,z)\ = \frac{1}{k}
\log \sum_{j\in (\Z/k\Z)^m} \left( \frac{\|\theta_j\|^2_{h_0^k}}{\|\theta_j\|^2_{h_1^k}}
\right)^t\frac{|\theta_j|^2_{h_0^k}}{\|\theta_j\|^2_{h_0^k}}.
\end{equation}
Our main theorem is to prove this term converges to $\phi_t(z)$ in the $C^{\infty}([0,1]\times M)$ topology.
But $$\phi_k(t,z)-\phi_t(z)=\frac{1}{k}
\log \sum_{j\in (\Z/k\Z)^m} \left( \frac{\|\theta_j\|^2_{h_0^k}}{\|\theta_j\|^2_{h_1^k}}
\right)^t\frac{|\theta_j|^2_{h_0^k}e^{-k\phi_t}}{\|\theta_j\|^2_{h_0^k}},$$
 denote $\rho _{k}(j,t)=\|\theta_j\|^2_{h_t^k}$ as the norming constant and denote $$R_k(j,t)= \frac{\rho _{k}(j,t)}{(\rho _{k}(j,0))^{1-t}(\rho _{k}(j,1))^t},$$
and as usual $h_t = e^{-\phi_t}h_0$, then we can rewrite
 \begin{equation}\phi_k(t,z)-\phi_t(z)=\frac{1}{k}\log \sum_{j\in (\Z/k\Z)^m}R_k(j,t)\frac{|\theta_j|^2_{h_t^k}}{\|\theta_j\|^2_{h_t^k}}.\end{equation}
Thus our goal is equivalent to prove this term goes to $0$ in the $C^{\infty}$ topology as $k \rightarrow \infty$.
In fact we prove the following result that implies Theorem \ref{ghjiuyum} immediately:
\begin{lemma}\label{oiut} With all assumptions and notations as above, we have:
$$\frac{1}{k} \log\sum_{j\in (\Z/k\Z)^m} R_k(j,t)\frac{|\theta_j|^2_{h_t^k}}{\|\theta_j\|^2_{h_t^k}}=mk^{-1}\log k+\log R_{\infty}(\mu_t,t)+ k^{-2}c_1(\mu_t,t)+\cdots $$ where $\mu_t(y)=\nabla \varphi_t(y)$,  $c_n(\mu_t,t)\in C^\infty(M\times [0,1])$ and periodic in $y$ variables for any fixed $t$ and $R_{\infty}$ is defined by (\ref{uyee}).
Furthermore, this expansion can be differentiated any number of times on both sides with respect to $t$ and $y$ (or $z$).
\end{lemma}

The proof of Lemma \ref{oiut} is a consequence of the following two facts
\begin{itemize}
\item Regularity: $R_k(j,t)$ admits the complete asymptotics with the leading term given by $R_{\infty}(x,t)$ evaluated at the point $x_0=-\frac{4\pi j}{k}$ (Lemma \ref{jhgf}).

\item The generalized Bernstein Polynomial $\sum_{j\in (\Z/k\Z)^m}f(-\frac{j}{k})\frac{|\theta_j|^2_{h^k}}{\|\theta_j\|^2_{h^k}}$ admits complete asymptotics for any  periodic function $f$ defined on $\R^m$ with period 1 (Lemma \ref{ddgsgs}).
\end{itemize}
\subsection{ Dedekind-Riemann sums }
In section \ref{dddsgg}, we prove the following generalized Bernstein Polynomial Lemma using the basic properties of theta functions and Weyl quantization:
\begin{lemma}\label{ddgsgs}Let $f(x) \in C^{\infty}(\R^m)$ and periodic in each variable with period $1$, let $h\in \hcal_0^{\Gamma}$, then we have the complete asymptotics:
\begin{equation}\label{bvc}\frac{1}{k^m}\sum_{ j\in(\Z/k\Z)^m}f(-\frac{j}{k})\frac{|\theta_j|^2_{h^k}}{\|\theta_j\|^2_{h^k}}=f(\mu)+k^{-1}b_1(\mu)+\cdots\end{equation}
where $\mu(y)=y+\nabla\psi(y)$ and $b_n(\mu)\in C^\infty(\R^m)$ for all $n\in \mathbb{N}$.\end{lemma}
The generalized Bernstein polynomial Lemma \ref{ddgsgs} has an application to Dedekind-Riemann sums for the periodic functions.
Results about the complete asymptotics of Dedekind-Riemann sums for the smooth functions with compact support over the polytope $P$ were obtained by Brion-Vergne, Guillemin-Sternberg and many others (cf. \cite{BV, GS}). For purposes of comparison, Theorem 4.2 of \cite{GS} states that for $f\in C^{\infty}_{0}(\R^n)$:
$$\frac{1}{k^m}\sum _{\alpha\in \Z^m \cap kP}f  (\frac{\alpha}{k})=\left(\sum_F\sum_{\gamma \in \Gamma_{F}^{\sharp}}\tau _{\gamma}\left(\frac{1}{k}\frac{\partial}{\partial h}\right)\int_{P_{h}}f(x)dx \right)\mid_{h=0}$$
where $\alpha$ is the lattice point in the $kth$ dilate of the polytope $kP$ and $P_h$ is a parallel dilate of $P$. We refer to \cite{GS} for more details.

Afterward,  Zelditch related the Bernstein polynomials to the Bergman kernel for the Fubini-Study metric on $\CP^1$, and generalized this relation to any compact \kahler toric manifold, implying many interesting results \cite{Z1}. To be more precise, let $(L,h)\rightarrow (M,\omega)$ be a toric Hermitian invariant line bundle over a \kahler toric manifold with associated moment polytope $P$, he proved the following complete asymptotics:
 \begin{equation}\sum_{\alpha\in \Z^m \cap kP}f(\frac{\alpha}{k})\frac{|s_{\alpha}|^2_{h^k}}{\|s_{\alpha}\|^2_{h^k}}=f(x)+k^{-1}\mathcal{L}_1f(x)+k^{-2}\mathcal{L}_2f(x)+\cdots \end{equation} where $f\in C_{0}^{\infty}(\R^m)$, each $\mathcal{L}_j$ is a differential operator of order $2j$, $s_{\alpha}$ is the orthogonal basis of $\Hk$ which in fact are monomials $z^{\alpha}$. Then the simple integration yields:
 \begin{equation}\frac{1}{k^m}\sum_{\alpha \in \Z^m\cap kP }f(\frac{\alpha}{k})=\int_{P}f(x)dx+ \frac{1}{2k}\int_{\partial P}f(x)dx+\frac{1}{k^2}\int_{P}\mathcal{L}_2f(x)dx+\cdots\end{equation}
In \cite{F}, this method is then generalized to the polyhedral set.

In section \ref{dddsgg}, we will first generalize the method in \cite{F,Z1} to Abelian varieties to get the Lemma \ref{ddgsgs}. If we take the integral over $M$ on both sides of (\ref{bvc}) and note $\sum_{j\in(\Z/k\Z)^m}f(\frac{j}{k})=\sum_{j\in(\Z/k\Z)^m}f(-\frac{j}{k})$, then we have the following Dedekind-Riemann sums for periodic functions:
\begin{cor}\label{iuu} Let $f(x) \in C^{\infty}(\R^m)$ and periodic in each variable with period 1, then:
\begin{equation}\label{bvsssc}\frac{1}{k^m}\sum_{j\in(\Z/k\Z)^m}f(\frac{j}{k})=\int _{[0,1]^m}f(x)dx+ k^{-1}\int _{[0,1]^m}b_1(x)dx+\cdots
\end{equation} where each $b_n(x)\in C^\infty(\R^m)$ and can be computed explicitly.
\end{cor}

  \subsection{Complete asymptotics of harmonic maps} A harmonic map between two Riemannian manifolds $(N_1, g_1)$ and $(N_2, g_2)$
is a critical point of the energy functional
$$E(f) = \int_{N_1} |df|^2_{g_1\otimes f^*g_2}dVol_{g_1}$$
on the space of smooth maps $f: N_1 \rightarrow N_2$. Note that this notion may
also be defined when the target manifold $(N_2, g_2)$ is an infinite-dimensional
weakly Riemannian manifold, e.g., $(\hcal, g_{\hcal})$. By a smooth map $f$ from $N$ to $\hcal$ we mean a function
$f\in C^\infty(N \times M)$ such that $f(q, \cdot) \in \hcal$ for each $q \in N$ (see Definition 1.1 in \cite{RZ}).

In \cite{RZ}, Rubinstein-Zelditch
proved that, in the toric case, the Dirichlet problem for
a harmonic map $\varphi: N \rightarrow \hcal(\mathbb{T}^m)$ of any compact Riemannian manifold  $N$ with
smooth boundary into $\hcal(\mathbb{T}^m)$ of toric invariant metrics  admits a smooth solution that
may be approximated in $C^2(N\times M)$ by a special sequence of harmonic maps
$\varphi_k : N \rightarrow \bcal_ k(\mathbb{T}^m) \subset \hcal(\mathbb{T}^m)$ into the subspaces $\bcal_ k(\mathbb{T}^m)$ of Bergman
metrics (Theorem 1.1 in \cite{RZ}). This generalized the work of Song-Zelditch in the case of geodesics, i.e., where $N = [0, 1]$.

In the spirit of \cite{RZ}, we consider the harmonic maps into the space of $\hcal_0^\Gamma$ of $\Gamma$-invariant Abelian metrics .
 Then we can prove that the approximation of the harmonic into $\hcal_0^\Gamma$ by the corresponding ones into $\bcal_k$ is still $C^\infty$. 

\begin{theo}\label{harmin}
Let $M$ be a principally polarized Abelian variety and let $L\rightarrow M$ be a principal polarization. Let
$(N, g)$ be a compact oriented smooth Riemannian manifold with smooth
boundary $\partial N$. Let $\psi: \partial N \rightarrow \hcal_0^\Gamma$ denote a fixed smooth map. There
exists a harmonic map $\varphi:N \rightarrow \hcal_0^\Gamma$ with $\varphi|_{\partial N} = \psi$ and harmonic maps
$\varphi_k : N \rightarrow \bcal_k$ with $\varphi_k |_{\partial N} = FS_k \circ Hilb_k(\psi)$, then we have the following $C^\infty$ complete asymptotics, $$\varphi_k=\varphi+mk^{-1}\log k+k^{-1}a_1+k^{-2}a_2+\cdots$$
where each $a_n$ is smooth and $a_1=\log K_\infty$ where $K_\infty$ is defined by (\ref{ddgssd}).
\end{theo}

The proof of Theorem \ref{harmin} is similar to the one in \cite{RZ}. In section \ref{testharmonic}, we will sketch the main steps of the proof for the model case.

\subsection{Final remarks}\label{test}
The $C^2$ convergence of Song-Zelditch for the toric varieties can be improved to the $C^{\infty}$ convergence for the Abelian varieties mainly because of the Regularity Lemma \ref{jhgf}: $R_k(j,t)$ admits complete asymptotics. But for the toric case, they do not know the existence of the complete asymptotics of $R_k(\alpha,t)$, where $\alpha$ is a lattice point in $P$ which is the image of the moment map of toric varieties $\nabla_{\rho} \varphi : M \rightarrow P$. In fact, they have the following lemma: $$(\frac{\partial}{\partial t})^n R_k(\alpha,t)=(\frac{\partial}{\partial t})^nR_{\infty}(\frac{\alpha}{k},t)+O(k^{-\frac{1}{3}})\,\,,0 \leq n\leq 2.$$ They can not prove the existence of complete asymptotics because they can not get the joint asymptotics of $k$ and $\alpha$ of the norming constant $\rho_{k}(\alpha)=\|s_{\alpha}\|_{h^k}^2$, where $s_{\alpha}$ is the holomorphic section of the invariant line bundle. Recall that the boundary of $P$ is the image of the points with isotropy group of $\mathbb{T}^n$, $1 \leq n\leq m$ under the moment map $\nabla_{\rho}\varphi$ and the boundary causes serious complications. To be more precise, they can rewrite $\rho_{k}(\alpha)$ as:
  $$\rho_{k}(\alpha)=\int_P e^{-k(u_{\varphi}(x)+\langle\frac{\alpha}{k}-x, \nabla u_{\varphi}(x)\rangle)}dx$$ where $u_{\varphi}$ is the symplectic potential defined on $P$, i.e., the Legendre transform of \kahler potential $\varphi$.  Note that the critical point of the phase is given by $\frac{\alpha}{k}$; thus they can get complete asymptotics by the stationary phase method when the point $\frac{\alpha}{k}$ is far away from the boundary of $P$. But they can not get joint asymptotics by this method when the point goes to the boundary $\partial P$ as $k \rightarrow \infty$ \cite{SoZ}.

   But in our Abelian case,  we do not have such disadvantage. There is a real torus $\Gamma=(S^1)^m$ action on the Abelian varieties. This action is free, i.e., there is no point with the isotropy group of $(S^1)^n$, $1\leq n\leq m$. In section \ref{dghdSg}, we will see that the gradient of the \kahler potential induces a map $\nabla \varphi_t=4\pi(y+\nabla \psi_t): M \rightarrow M/\Gamma$ which is in fact a Lie group valued moment map for any fixed $t$. The image of $\nabla \varphi_t $ is $M/\Gamma$ which has no boundary. There is another way to look at this, in section \ref{gnbvgi}, we rewrite  $\rho_{k}(j)=\|\theta_j\|_{h^k}^2$ as an integral over the universal covering space $\R^m$ (\ref{dhghgndh}):
   $$\rho_{k}(j)=e^{-2\pi \frac{j^2}{k}}\int_{\R^m}e^{-k\pi (-u(x)+\langle x+\frac{4\pi j}{k}, \nabla u(x) \rangle)}dx$$
  where $u(x)$ is defined by Legendre transform of $\varphi$, thus we can apply the stationary phase method to this integral everywhere.

 For example, in section \ref{jhnmb}, we can get identity (\ref{fdsfs}) which is the exact formula for $\rho_k(j)=\|\theta_j\|_{h^k}^2$. If we assume $\psi\equiv 0$, i.e., we choose the flat metric over the Abelian variety, then $\|\theta_j\|^2_{h^k}$ will be a constant independent of $j$, i.e., the joint complete asymptotics of $\rho_k(j)$(which is in fact a constant) exist for any $j$ as $k \rightarrow \infty$. This is totally different from the toric case. For example, consider $(\mathbb{CP}^1, \omega_{FS})$ with Fubini-Study metric, then $\|z^{\alpha}\|^2_{h^k_{FS}}=$${k \choose \alpha}$$^{-1}$, but as proved in \cite{SoZ1}, for any $\alpha\in [k^{-\frac{3}{4}}, 1-k^{-\frac{3}{4}}]$, by stationary phase method:
   $${k \choose k\alpha} \sim \frac{1}{\sqrt{2\pi k \alpha (1-\alpha)}}e^{(\alpha \log \alpha +(1-\alpha)\log(1-\alpha))}$$
   Then it's easy to see that the asymptotics are highly non-uniform as $\alpha \rightarrow 0$ or $\alpha \rightarrow 1$, where $0$ and $1$ are two boundary points of the moment polytope $[0,1]$ of $\mathbb{CP}^1$.
\bigskip

\noindent{\bf Acknolwedgements:} The author would like to thank Prof. S. Zelditch for his support of this project. He would like to thank Dr. Z. Wang for many helpful discussions. Many thanks go to Dr. Y. A. Rubinstein for discussing the problem and sharing many of his fresh ideas, for reading the first version line by line, pointing out mistakes and typos and giving many suggestions about how to write the article. The author also would like to thank the referee for many helpful comments in the original version. This paper will never come out without their helps.

\section{Abelian varieties and Theta functions}\label{jhnmb}
In this section, we will review some basic properties of principally polarized
Abelian varieties and theta functions, we mainly follow \cite{FMN}, refer to \cite{GH, Mu} for more details.

Let $V$ be a $m$-dimensional complex vector space and $\Lambda \cong \Z^{2m}$ a maximal
lattice in $V$ such that the quotient
$M = V/\Lambda$ is an Abelian variety, i.e., a complex
torus which can be holomorphically embedded in projective space.  We
assume that $M$ is endowed with a principal polarization, then we can always find a basis $\lambda_1,...,\lambda_{2m}$ for $\Lambda$, such that $\lambda_1,...,\lambda_{m}$ is a basis of $V$ and
$$\lambda _{m+\alpha}=\sum_{\beta =1}^mZ_{\beta \alpha}\lambda_{\beta}, \,\, \alpha=1,...,m$$ where $Z=(Z_{\alpha\beta})_{\alpha,\beta=1}^{m}$is a $m\times m$ matrix satisfies $Z^T=Z$ and $Im Z>0$. Conversely,
principally polarized Abelian varieties are parametrized by such matrices.

Let $x_1,...,x_m,y_1,..., y_m$ be the coordinates on $V$ which are dual to the generators $\lambda_1,...,\lambda_{2m}$ of the lattice $\Lambda$. Then $x_{\alpha}$ and $y_{\alpha}$ can also be considered as periodic coordinates in $M$, and are related to the complex ones by:
\begin{equation}\label{jhgui}z_{\alpha}=x_{\alpha}+\sum _{\beta=1}^mZ_{\alpha \beta}y_{\beta}\,\,\,,\bar z_{\alpha}=x_{\alpha}+\sum _{\beta=1}^m \bar Z_{\alpha \beta}y_{\beta}.\end{equation}

Let $L \rightarrow M$ be the holomorphic line bundle, if we further assume $L$ is a principal polarization of $M$, then the first Chern class $c_1(L)$ is given by: \begin{equation}\begin{array}{lll} \label{poiubvy}\omega_0
& = & \sum_{\alpha=1}^m dx_{\alpha}\wedge dy_{\alpha}
 \\ && \\
& =& \frac{\sqrt{-1}}{2}\sum_{\alpha, \beta}(Im Z)^{\alpha \beta}dz_{\alpha} \wedge d \bar z_{\beta}. \end{array} \end{equation} The space $\Hk$ is naturally isomorphic
with the space of holomorphic functions $\theta$ on $V$ satisfying:
\begin{equation}\theta(z+\lambda_{\alpha})=\theta(z), \,\,\, \theta(z+\lambda_{m+\alpha})=e^{-2k\pi iz_{\alpha}-k\pi i Z_{\alpha \alpha}}\theta(z).\end{equation}
In fact, these theta functions are in form \cite{FMN}:
$$\theta(z)=\sum_{l\in (\Z/k\Z)^m}a_l\theta_l( z,\Omega),$$
where \begin{equation}\label{dsbn}\theta_l(z,\Omega )=\sum_{n\in \Z^m}e^{\pi i(l+kn) \frac{Z}{k}(l+kn)^T}e^{2\pi i(l+kn)\cdot z}, \,\,\, l\in  (\Z/k\Z)^m .\end{equation}
In particular, dim $\Hk$$=k^m$.

Now consider the hermitian metric $h$ defined on $L$, $h$ should be a positive $C^{\infty}$ function of $z$ satisfying:
\begin{equation}\label{iuytreeth}h(z)|\theta(z)|^2=h(z+\lambda)|\theta(z+\lambda)|^2\end{equation} for any $\lambda \in \Lambda$; thus \begin{equation}\label{iuyth} h(z+\lambda_{\alpha})=h(z)\,\,, h(z+\lambda_{m+\alpha})=|e^{2\pi i z_{\alpha}}|^2|e^{\pi i Z_{\alpha \alpha}}|^2h(z).\end{equation}
Conversely, any such function $h$ defines a metric on $L$.
\bigskip

For simplicity, we first consider the Abelian variety $M=\C^m/\Lambda$, where $\Lambda=\Z^m+i\Z^m$. Write $z=x+iy$, where $x$ and $y \in \R^m$ and can be viewed as the periodic coordinate of $M$. Let $L \rightarrow M$ be a principal polarization of $M$, then by formula (\ref{dsbn}), the global holomorphic section of $H^0(M,L)$ is given by the following Riemann theta functions:
\begin{equation}\label{oiuy} \theta(z)=\sum_{n\in \Z^m}e^{-\pi n^2 +2\pi i n\cdot z}, \end{equation}
where $n^2=n_1^2+\cdots+n_m^2$ and $n\cdot z=n_1z_1\cdots+n_mz_m$. And  the global holomorphic section of $\Hk$ is given by:
\begin{equation}\label {a}\theta_j(z)=\sum_{n \in \Z^m}e^{-\pi \frac{(j+kn)^2}{k} +2\pi i(j+kn)\cdot z} ,\,\,\, j\in(\mathbb{Z}/k\mathbb{Z})^m \end{equation}
with dim$\Hk$$=k^m$. Furthermore, $\theta_j(z)$ are holomorphic functions over $\C^m$ and satisfy the following quasi-periodicity relations:
\begin{equation}  \theta_j(z_{\alpha}+1)=\theta_j(z_{\alpha}) ,\,\,\,\, \theta_j(z_{\alpha}+i) = e^{-2\pi ikz_{\alpha}+k\pi}\theta_j(z_{\alpha}).
 \end{equation}
Now define the hermitian metric on $L$ as  $$h_t= e^{-2\pi y^2-4\pi\psi_t(y)},$$
where $\psi_t(y)$ is a smooth and periodic function of $y\in \R^m$ with period $1$. It's easy to check $h_t$ satisfies conditions (\ref{iuyth})
$$h_t(z_{\alpha}+1)=h_t(z_{\alpha})\,\,,h_t(z_{\alpha}+i)=|e^{2\pi i z_{\alpha}}|^2e^{2\pi}h_t(z_{\alpha}),$$ thus $h_t$ is a well defined hermitian metric on $L$.

Now in our case, the natural Hermitian inner product (\ref{dsldd}) defined on the space $\Hk$ reads:
\begin{equation}\label {acbn}\langle \theta_l,
\theta_j\rangle_{h_t^k}= \int_{M} \theta_l(z) \overline{ \theta_j(z)}e^{-2k\pi y^2-4k\pi\psi_t(y)}\frac{\omega_{h_t}^m}{m!} ,\end{equation}
where the volume form $\frac{\omega_{h_t}^m}{m!}=  (4\pi)^m\det(I+\hess \psi_t)dxdy$.

\begin{prop}\label{dfghhg}$\left\{\theta_j\,\,,j \in(\Z/k\Z)^m \right \}$ forms an orthogonal basis of $\Hk$ with respect to the Hermitian inner product defined by (\ref {acbn}). \end{prop}
\begin{proof} By definition,
\begin{equation} \begin{array}{lll} \label{poiuy}\langle\theta_l, \theta_j\rangle_{h_t^k}
& = & (4\pi)^m \int _{[0,1]^m} \int _{[0,1]^m}[\sum_{n \in\Z^m}e^{-\pi \frac{(l+kn)^2}{k} +2\pi i(l+kn)\cdot z} ] \cdot
\\ && \\
&  & [\sum_{p \in \Z^m}e^{-\pi \frac{(j+kp)^2}{k} -2\pi i(j+kp)\cdot \bar z}] e^{-2k\pi y^2-4k\pi\psi_t(y)}\det(I+\hess \psi_t)dxdy  \\ && \\
& = & (4\pi)^m  [\sum_{n\in\Z^m}\sum_{p\in\Z^m}\int _{[0,1]^m}e^{ 2\pi i (l+kn-j-kp)\cdot x}dx] \cdot \\ && \\
& &  [\int _{[0,1]^m}e^{-\pi \frac{(l+kn)^2+(j+kp)^2}{k}-2\pi (l+kn+j+kp)\cdot y-2k\pi y^2-4k\pi\psi_t}\det(I+\hess \psi_t) dy]\end{array} \end{equation}
For the first integral, if $l_{\alpha}+kn_{\alpha}=j_{\alpha}+kp_{\alpha}$, i.e., $l_{\alpha}-j_{\alpha}=0 \mod k$, then $$\int _{[0,1]}e^{ 2\pi i (l_{\alpha}+kn_{\alpha}-j_{\alpha}-kp_{\alpha})x_{\alpha}}dx_{\alpha}=1,$$ otherwise, it's 0.
Since $1 \leq l_{\alpha}, j_{\alpha}\leq k$, hence  $l_{\alpha}+kn_{\alpha}=j_{\alpha}+kp_{\alpha}$ iff $l_{\alpha}=j_{\alpha}$ and $p_{\alpha}=n_{\alpha}$; thus the first integral is nonzero iff $l=j$ and $n=p$. Then equation (\ref{poiuy}) becomes:
 $$\langle\theta_l, \theta_j\rangle_{h_t^k} =(4\pi)^m \delta_{l,j}\sum_{n\in \Z^m}\int _{[0,1]^m}e^{-2k\pi(\frac{j}{k}+n+y)^2}e^{-4k\pi \psi_t(y)}\det(I+\hess \psi_t) dy.$$
 Hence, we can see that $\left\{ \theta_j \,\,, j\in (\Z/k\Z)^m\right\}$ forms an orthogonal basis of $\Hk$. \end{proof} Furthermore, we have: \begin{equation} \label{fdsfs} \begin{array}{lll} \|\theta_j\|^2_{h_t^k}&=&(4\pi)^m \sum_{n\in \Z^m}\int _{[0,1]^m}e^{-2k\pi(\frac{j}{k}+n+y)^2}e^{-4k\pi \psi_t(y)}\det(I+\hess \psi_t) dy
  \\ && \\
&= &
 (4\pi)^m  \int _{\R^m}e^{-2k\pi (y+\frac{j}{k})^2}e^{-4k\pi \psi_t(y)}\det(I+\hess \psi_t) dy.\end{array} \end{equation}
In the last step, we change variable $y\rightarrow y+n $ and use the fact that $\psi_t(y)$ is a smooth and periodic function with period 1. In fact, this integral is taken over the universal covering space $\R^m$.

\section{Regularity lemma}
\subsection{$\Gamma$-invariant metrics and geodesics}\label{dghdSg}
In this subsection, we recall some basic properties of the space $\hcal^{\Gamma}_0$ of $\Gamma$-invariant \kahler metric proved in \cite{D}.

 Now consider $M=\C^m/\Lambda$ where $\Lambda=\Z^m+i\Z^m$, we write each point in $M$ as $z=x+iy$, where $x$ and $y \in \R^m$ and can be considered as periodic coordinate in $M$. Let $\omega_0=\frac{\pi \sqrt{-1}}{2} \sum_{\alpha=1}^m dz_{\alpha} \wedge d\bar z_{\alpha}$ be the flat metric with associated local \kahler potential $2\pi y^2$. The group $\Gamma=(S^1)^m$ acts on $M$ via translations in the Langrangian subspace $\R^m \subset \C^m$, and this induces an isometric action of $\Gamma$ on the space $\hcal$ of \kahler metrics on $M$; so $\hcal_0^{\Gamma}$ of $\Gamma$-invariant metrics is totally geodesic in $\hcal$. Furthermore, $\hcal_0^{\Gamma}$ can be viewed as the set of functions:
     \begin{equation}\hcal_0^{\Gamma}\ = \
\{\psi\in C^{\infty}_{\Gamma} (M) : \omega_\psi\ = \ \omega_0+ \frac{\sqrt{-1}}{2}\partial \bar \partial \psi>0\}.\end{equation} In fact, a function invariant under the action of $\Gamma$ is independent of $x$; thus it descends to a smooth function on $M/\Gamma$, i.e., smooth and periodic function with period $1$ defined on $y\in \R^m$.

 The crucial point about $\hcal_0^{\Gamma}$ is: Given any two points $\varphi_0$ and $\varphi_1$ in  $\hcal_0^{\Gamma}$, there exists a smooth geodesics $\varphi_t(y)$ in $\hcal_0^{\Gamma}$ joining them. To be more precise, in the local coordinate, the geodesic is given by the path $\varphi_t(z)=2\pi y^2+ 4\pi \psi_t(y)$ satisfying the condition: \begin{equation}\label{iupoj} \ddot{\varphi} -\frac{1}{2}|\nabla \dot{\varphi}|^2_{\omega_{\psi}}=0\end{equation} Moreover, $\hess \varphi_t=I+\hess \psi_t>0$ because of the positivity of \kahler form; thus $\varphi_t$ is a convex function on $\R^m$.  Then the Legendre transform of $\varphi_t(y)$ \begin{equation}\label{cx}u_t(\mu)=\mu\cdot y-\varphi_t(y) \end{equation}is well defined
where \begin{equation}\label{css}\mu=\nabla \varphi_t=4\pi(y+\nabla \psi_t(y)).\end{equation} For any fixed $t$, the map $\mu(y,t)=\nabla \varphi_t: \R^m \rightarrow \R^m$ and also induces a map $\mu : M \rightarrow M/\Gamma$ which is an example of a Lie group valued moment map. Following the same proof in \cite{D, G, R}, we have:
\begin{prop}\label{jghnbv}$u(t,\mu)$ is linear along the geodesic (\ref{iupoj}). \end{prop}

According to this Proposition, we can solve equation (\ref{iupoj}) in $\hcal^{\Gamma}_0$ as follows: given any two \kahler potential $\varphi_0$ and $\varphi_1$, make the Legendre transform $u_0=\mathcal{L}\varphi_0$ and $u_1=\mathcal{L}\varphi_1$, then \begin{equation}\label{linar}u_t=(1-t)u_0+tu_1\end{equation} solve equation $\ddot{u}=0$; thus the inverse of Legendre transform  \begin{equation}\varphi_t=\mathcal{L}^{-1}u_t\end{equation} will solve equation (\ref{iupoj}) which is $C^{\infty}$.
\subsection{Regularity Lemma } \label{gnbvgi}
Denote $u(t, \mu)=\mathcal{L} \varphi_t(y)$ as the Legendre transform of $\varphi_t(y)$ for any fixed $t$. 
By properties of Legendre transform, we have: \begin{equation}\label{sdsds}y=\nabla_\mu u,\end{equation} \begin{equation}\label{poijk} \frac{\partial y}{\partial \mu}=(\hess_y \varphi)^{-1}(y)=\frac{1}{4\pi}(1+\hess \psi_t)^{-1}(y)>0.\end{equation}

Let $\rho _{k}(j,t)=\|\theta_j\|^2_{h_t^k}$ denote the norming constant. Define \begin{equation} \label{utye} R_k(j,t)= \frac{\rho _{k}(j,t)}{(\rho _{k}(j,0))^{1-t}(\rho _{k}(j,1))^t},\end{equation}
\begin{equation}\label{uyee} R_{\infty}(\mu,t)=(\frac{\det \hess_\mu u}{(\det \hess_\mu u_{0})^{1-t}(\det \hess_\mu u_{1})^t})^{1/2} .\end{equation} 
We have following regularity lemma:
\begin{lemma}\label{jhgf} We have the following complete asymptotics, 
$$R_{k}(j,t)=  R_{\infty}(\mu,t)(1+k^{-1}a_1+k^{-2}a_2+ \cdots +k^{-\nu}a_{\nu})|_{\mu=-\frac{4\pi j}{k}}+O(k^{-\nu-1})$$
where $\nu$ is any positive integer and $O$ symbol is uniform in $t$. Moreover, $R_{\infty}(\mu,t)$ and each $a_{\nu}$ are smooth functions of $(\mu,t)$ and $4\pi$- periodic in $\mu$ for any fixed $t$.
\end{lemma}
\begin{proof}
The periodicity of $R_{\infty}(\mu,t)$ is easy to see since the map $\mu: y\rightarrow \nabla_y \phi$ induces a map from $M$ to $M/\Gamma$, thus all functions in $\mu$ variables will be periodic.

First from (\ref{poijk}), we have,
\begin{equation} \label{uvbctye} d\mu=(4\pi)^m\det(I+\hess \psi_t) dy.\end{equation}

Now plug (\ref{cx}), (\ref{sdsds}) and (\ref{uvbctye}) into (\ref{fdsfs}), then we can rewrite the norming constant $\rho_{k}(j,t)$ as  \begin{equation}\label{dhghgndh}\rho_{k}(j,t)=e^{-2\pi \frac{j^2}{k}}\int_{\R^m}e^{-k(\mu \cdot \frac{\partial u}{\partial \mu}-u+\frac{4\pi j}{k}\cdot \frac{\partial u}{\partial \mu})}d\mu.\end{equation} Hence, by definition, we can rewrite $R_k(t,j)$ as
$$R_k(j,t)=\frac{ \int _{\R^m}e^{-k\pi (\mu\cdot \frac{\partial u}{\partial \mu}-u+\frac{4\pi j}{k}\cdot\frac{\partial u}{\partial \mu})}d\mu}{  (\int _{\R^m}{e^{-k\pi (\mu \cdot \frac{\partial u_{0}}{\partial \mu}-u_0+\frac{4\pi j}{k}\cdot\frac{\partial u_0}{\partial \mu})}}d\mu)^{1-t}(\int _{\R^m}{e^{-k\pi (\mu\cdot \frac{\partial u_{1}}{\partial \mu}-u_0+\frac{4\pi j}{k}\cdot\frac{\partial u_1}{\partial \mu})}}d\mu)^t}.$$
Recall the stationary phase method (Theorem 7.7.5 in \cite{H}):
\begin{equation}\label{polk}\int u(x)e^{ik \Psi(x)}dx=\frac{e^{ik\Psi(x)}}{\sqrt{\det(k\hess \Psi (x)/2\pi i)}}
\sum_{\lambda= 0}^{\infty} k^{ - \lambda}  L_\lambda u|_{x=x'} \end{equation}
where $x'$ is the critical point of $\Psi$, Im$ \Psi \geq 0$ and $L_\lambda$ is a differential operator of order $2\lambda$.

Note that in \cite{H}, $u(x)$ is assumed to has compact support, but in fact this formula is true for any $u(x)\in C^\infty(\R^m)$. The strategy is to choose a cut-off function  $\chi$ in a neighborhood of $x'$ and rewrite the amplitude $u$ to be $\chi u+(1-\chi) u$, then separate the integration into two parts correspondingly. To the integration with the amplitude $\chi u$, we use the formula of stationary phase method directly; to the second part, by Theorem 1.1.4 in \cite{So}, is $O(k^{-\infty})$.

To our case, note that the hypotheses of \cite{H} are
satisfied since we can add some constant to ensure our phase function has positive imaginary part. Now the critical point of the phase $\Psi=\mu\cdot\frac{\partial u}{\partial \mu}-u+\frac{4\pi j}{k}\cdot\frac{\partial u}{\partial \mu}$ satisfies: $( \mu'+\frac{4\pi j}{k})\cdot \hess u=0$.  Thus the critical point of the phase is given by $\mu'=-\frac{4\pi j}{k}$ since the matrix $\hess u>0$ .
And the Hessian of the phase at the critical point is $\hess \Psi|_{\mu=\mu'}=\hess u(\mu',t)>cI$. Thus by the formula of the stationary phase method, we have
\begin{equation} \begin{array}{lll}\label{t}& &\int _{\R^m}e^{-k (\mu\cdot\frac{\partial u}{\partial \mu}-u+\frac{4\pi j}{k}\cdot\frac{\partial u}{\partial \mu})}d\mu\\ && \\
& = & k^{-\frac{m}{2}}(e^{-k (\mu\cdot\frac{\partial u}{\partial \mu}-u+\frac{4\pi j}{k}\cdot\frac{\partial u}{\partial \mu})}\sqrt{\det \hess u})(1+k^{-1}L_1(t,\mu)+k^{-2}L_2(t,\mu)\cdots)|_{\mu'=-\frac{4\pi j}{k}} \\ && \\
& = & k^{-\frac{m}{2}}(e^{k u}\sqrt{\det \hess u})(1+k^{-1}L_1(t,\mu)+k^{-2}L_2(t,\mu)\cdots)|_{\mu'=-\frac{4\pi j}{k}}. \end{array} \end{equation}
where each $L_{\lambda}$ is a smooth function of $(\mu,t)$ and $4\pi$-periodic in $\mu$ for any fixed $t$. 


Now we can get the following expression of $R_k(j,t)$ by expanding each term in denominator and numerator:
\begin{equation} \begin{array}{lll}R_k(j,t)& = & e^{k\pi(u-(1-t)u_0-tu_1)}(\frac{\det \hess u}{(\det \hess u_{0})^{1-t}(\det \hess u_{1})^t})^{1/2}\frac{1+k^{-1}L_1(t,\mu)+\cdots}{(1+k^{-1}L_1(0,\mu)+\cdots)^{1-t}(1+k^{-1}L_1(1,\mu)+\cdots)^t}|_{\mu=-\frac{4\pi j}{k}} \\ && \\
& = & R_{\infty}(\mu,t)(1+k^{-1}a_1+k^{-2}a_2+ \ldots +k^{-\nu}a_{\nu})|_{\mu=-\frac{4\pi j}{k}}+O(k^{-\nu-1}) . \end{array} \end{equation}
In the last step, we plug in the identity (\ref{linar}). Then we apply the Taylor expansion $(1+x)^\gamma=1+\gamma x+ \cdots$ to the term $(1+k^{-1}L_1(t,\mu)+\cdots)(1+k^{-1}L_1(0,\mu)+\cdots)^{t-1}(1+k^{-1}L_1(1,\mu)+\cdots)^{-t}$, choosing $\gamma$ as $t-1$ and $-t$. If we expand these three terms completely, we will  get the complete asymptotics, and we can compute each term step by step. For example, the first term is $1$ and the second term is $k^{-1}(L_1(t,\mu)-(1-t)L_{1}(0,\mu)-tL_1(1,\mu))$. Moreover, $a_{\nu}$ is a polynomial of  $t$ and $L_{\lambda}(t,\mu)$ for some $\lambda$, hence each $a_{\nu}$ is smooth and uniformly bounded on $[0,1]\times M$, and periodic for any fixed $t$ . Furthermore, if we combine this with the fact that $R_{\infty}(\mu,t)$ is uniformly bounded, then the error term $R_{\infty}(\mu,t)a_{\lambda+1}$ is uniformly bounded, i.e., the symbol $O$ is uniformly bounded.
\end{proof}

\section{Generalized Bernstein Polynomial} \label{dddsgg}
In this section, we will prove the Lemma \ref{ddgsgs}. We first introduce the definition and some basic properties of the Bergman kernel, refer to \cite{SZ,Z1,Z2} for more background.

Let $(L,h) \to (M,\omega)$ be a positive holomorphic line bundle
over a compact \kahler manifold of complex dimension $m$. We assume
$\omega=-\frac{\sqrt{-1}}{2}\partial\bar{\partial} \log |{s(z)}|_{h}^{2}$,
where $s(z)$ is a local holomorphic frame. We now define the Bergman kernel as the orthogonal projection from the $L^2$ integral sections to the holomorphic sections:
\begin{equation}\Pi_{k}: L^{2}(M,L^{k})\rightarrow H^{0}(M,L^{k}).\end{equation}
Furthermore, if $\left\{s_{j}^k\right \}_{j=0} ^{d_k}$ is an orthonormal basis of  $H^{0}(M,L^{k})$ with respect to the inner product defined by (\ref{dsldd}), then \begin{equation}\label{fghf}\Pi_{k}(z,w)=\sum_{j=0}^{d_k}s_j^k(z)\otimes \overline{s_j^k(w)}, \end{equation}
where $d_k+1$=dim $H^{0}(M,L^{k})$.
The following holds for any $m$-dimensional \kahler manifold \cite{BBS,BS, SZ}:
\begin{prop}For any $C^{\infty}$ positive hermitian line bundle $(L,h)$, we have:
\begin{equation}\label{sfgs} \Pi_k(z,w)= e^{k(\phi(z, \bar{w})-\frac{1}{2 }
(\phi(z)+\phi(w)))}A_{k}(z,w) +O( k^{-\infty}),
\end{equation}
where $\phi$ is the smooth local \kahler potential for $h$, $\phi(z, \bar{w })$ is the almost analytic extension of $\phi(z)$ and $A_{k}(z,w)= k^m(1+k^{-1}a_{1}(z,w)+\cdots)$ a semi-classical symbol of order $m$.
\end{prop}

\bigskip

Now we turn to the proof of Lemma \ref{ddgsgs}:
\begin{proof}Assume $M=\C^m/\Lambda$ where $\Lambda=\Z^m+i\Z^m$ and $L\rightarrow M$ is a principal polarization of $M$. Choose \kahler potential $\varphi(y)=2\pi y^2+4\pi\psi(y)$ as before. From Proposition \ref{dfghhg}, $\left\{\theta_j, j\in (\Z/k\Z)^m \right \}$ forms an orthogonal basis of $\Hk$ with respect to the Hermitian inner product defined by (\ref {dsldd}); thus by formula (\ref{fghf}), the Bergman kernel is given by:
\begin{equation}\label{kjhl}\Pi_k(z,w) = \sum_{j \in (\Z/k\Z)^m}\frac{\theta_j(z) \overline{\theta_j(w)}e^{-\frac{k\varphi(Im z)}{2}-\frac{k\varphi(Im w)}{2}}}{\|\theta_j\|^2_{h^k}}.\end{equation}
  For any function $f(x)\in C^\infty(\mathbb{T}^m)$, we can define the following translation operator
 $U: f(x)  \rightarrow f(x-\frac{1}{k})$ on the universal covering space.
If we consider this operator acting on the vector space $\Hk$ of holomorphic theta functions, then we have the following Weyl quantization \cite{K,KR}:
\begin{equation}Op_k(f)=\sum_{n\in \Z^m}\widehat{f}(n)U^{n},\end{equation} where $\widehat{f}(n)$ is the Fourier coefficients of $f$.
Now apply $U$ to theta functions:
$$\theta_j(z)=\sum_{n\in \Z ^m}e^{-\pi \frac{(j+kn)^2}{k} +2\pi i(j+kn)\cdot z} ,$$
then for any $x \in \R^m$, it's easy to see that:
\begin{equation}\label{kjhbv}U (\theta_{j}(z+x))=e^{-2\pi i \frac{j}{k}} \theta_j(z+x),\end{equation} where $e^{-2\pi i \frac{j}{k}}=e^{-2\pi i \frac{j_1}{k}}\cdots e^{-2\pi i \frac{j_m}{k}}$.
 Next apply $Op_k(f)$ to theta functions, we have: \begin{equation}\label{kjdghl}Op_k(f)\theta_j(z+x)=\left(\sum_{n\in \Z^m}\widehat{f}(n)e^{-2\pi i \frac{j}{k}\cdot n}\right) \theta_j(z+x)=f(-\frac{j}{k})\theta_j(z+x).\end{equation}
Now apply this operator to the Bergman kernel off the diagonal (\ref{kjhl}), we have:
\begin{equation} \begin{array}{lll}Op_k(f)\Pi_k(z+x,w)|_{x=0} & = &Op_k(f)\sum_{j\in (\Z/k\Z)^m}\frac{\theta_j(z+x) \overline{\theta_j(w)}e^{-\frac{k\varphi(Im z)}{2}-\frac{k\varphi(Im w)}{2}}}{\|\theta_j\|^2_{h^k}}|_{x=0}\\ && \\&= & \sum_{j\in (\Z/k\Z)^m}f(-\frac{j}{k})\frac{\theta_j(z) \overline{\theta_j(w)}e^{-\frac{k\varphi(Im z)}{2}-\frac{k\varphi(Im w)}{2}}}{\|\theta_j\|^2_{h^k}}.
  \end{array}
  \end{equation}
  Here we use the fact that $\varphi (Im(z+x))=\varphi(Imz)=\varphi(y)$.
Now choose $z=w$, we have:
\begin{equation}\label{iuy}\frac{1}{k^m}\sum_{j\in (\Z/k\Z)^m}f(-\frac{j}{k})\frac{|\theta_j|_{h^k}}{\|\theta_j\|^2_{h^k}}=\frac{1}{k^m} Op_k(f)\Pi_k(z+x,z)|_{x=0}.\end{equation}
 Now we get the complete asymptotics of $\Pi_k(z+x,z)$ as follows: by assumption, our \kahler potential only depends on $y=$Im$z$, i.e.,  $\varphi(z)=\varphi(y)=\varphi(\frac{z-\bar z}{2i})$, thus the almost analytic extension of $\varphi$ is given by \begin{equation}\label{dsv}\varphi(z,\bar w)=\varphi(\frac{z-\bar w}{2i}).\end{equation} Hence, formula (\ref{sfgs}) reads:
 \begin{equation}\label{ddgdewd} \begin{array}{lll}\Pi_k(z+x,z) &= & e^{k(\varphi(z+x,\bar z)-\frac{1}{2 }
(\varphi(z+x)+\varphi(z)))}A_{k}(z+x,z)  \\ && \\&= &  e^{k(\varphi(z+x,\bar z)-\varphi(z))}A_{k}(z+x,z), \end{array}
  \end{equation}
  where  $A_{k}(z+x,z)= k^m(1+k^{-1}a_{1}(z+x,z)+\cdots)$.
  In the last step, we use the fact that $\varphi(z+x)=\varphi(z)=\varphi(Im z)$ again.

 Now apply the operator $\frac{1}{k^m} Op_k(f)$ on both sides of (\ref{ddgdewd}),
\begin{equation}\label{ddddgd} \begin{array}{lll} \frac{1}{k^m} Op_k(f)\Pi_k(z+x,z)|_{x=0} &= &\frac{1}{k^m} \sum_{n\in \Z^m} \hat f(n)U^{n}\Pi_k(z+x,z)|_{x=0} \\ && \\&= & \frac{1}{k^m}  \sum_{n\in \Z^m} \hat f(n)\Pi_k(z-\frac{n}{k},z) \\ && \\&= & \frac{1}{k^m}  \sum_{n\in \Z^m}  \hat f(n) e^{k(\varphi(z-\frac{n}{k},\bar z)-\varphi(z))}A_{k}(z-\frac{n}{k},z) \\ && \\&= & \frac{1}{k^m} \sum_{n\in \Z^m} \hat f(n) e^{k(\varphi(y- \frac{n}{ 2ik})-\varphi(y))}A_{k}(z-\frac{n}{k},z). \end{array}
  \end{equation}
  In the last step, by identity (\ref{dsv}), the almost analytic extension $\varphi(z-\frac{n}{k},\bar z)=\varphi(\frac{z-\frac{n}{k}-\bar z}{2i})=\varphi(y- \frac{n}{ 2ki})$.

To the last equation in (\ref{ddddgd}), if we apply the Taylor expansion to $e^{k(\varphi(y- \frac{n}{ 2ik})-\varphi(y))}$ and use the complete asymptotic of  $A_{k}(z-\frac{n}{k},z)= k^m(1+k^{-1}a_{1}(z-\frac{n}{k},z)+\cdots)$, we will get the complete asymptotic of $Op_k(f)\Pi(z+x,z)|_{x=0}$.  For example, we can compute the leading term as follows: first, $e^{k(\varphi(y- \frac{n}{ 2ik})-\varphi(y))}=e^{ -\nabla\varphi\cdot \frac{n}{ 2i}+O(k^{-1})}=e^{ -\nabla\varphi\cdot \frac{n}{ 2i}}+O(k^{-1})$; second, $\frac{1}{k^m} A_{k}(z-\frac{n}{k},z)= 1+O(k^{-1})$. Hence the leading term is given by, \begin{equation}\sum_{n\in \Z^m}\hat f(n)e^{ -\nabla\varphi\cdot \frac{n}{ 2i}} =f(\frac{\nabla \varphi}{4\pi})=f(\mu),\end{equation}
where $\mu=y+\nabla\psi$.  Hence, we can get the complete asymptotics step by step if we further expand $e^{k(\varphi(y- \frac{n}{ 2ik})-\varphi(y))}$ and $A_k$.  \end{proof}

 As a remark, if we replace $f$ and $h$ to be a path of smooth periodic function $f_t$ and any path $h_t$ in $\hcal_0^{\Gamma}$, then the lemma still holds with the leading term $f_t(\mu)$. Furthermore, we can differentiate the complete asymptotics with respect to $t$ on both sides.
 \section{$C^{\infty}$ convergence of Bergman geodesics} \label{dsss}
In this section, we will apply the Regularity Lemma and the generalized Bernstein Polynomial Lemma to prove Lemma \ref{oiut}:
\begin{proof}We first apply Lemma \ref{jhgf}, and denote $A_{\nu}(\mu,t)\sim R_{\infty}(\mu,t)a_{\nu}(\mu)$, then $A_{\nu}(\mu,t)$ is periodic in $\mu$ since $R_{\infty}(\mu,t)$ and $a_{\nu}(\mu)$ are periodic.
Then:
 $$\begin{array}{lll}& & \sum_{j\in (\Z/k\Z)^m}R_k(j,t)\frac{|\theta_j|^2_{h_t^k}}{\|\theta_j\|^2_{h_t^k}} \\ && \\&\sim & \sum_{j\in (\Z/k\Z)^m} R_{\infty}(\mu,t)(1+k^{-1}a_1+k^{-2}a_2+ \cdots )|_{\mu=-\frac{4\pi j}{k}}\frac{|\theta_j|^2_{h_t^k}}{\|\theta_j\|^2_{h_t^k}}
 \\ && \\
& \sim & \sum_{j\in (\Z/k\Z)^m} R_{\infty}(-\frac{4\pi j}{k},t)\frac{|\theta_j|^2_{h_t^k}}{\|\theta_j\|^2_{h_t^k}}+\frac{1}{k} \sum_{j\in (\Z/k\Z)^m}A_{1}(-\frac{4\pi j}{k},t)\frac{|\theta_j|^2_{h_t^k}}{\|\theta_j\|^2_{h_t^k}}+\cdots .
\end{array} $$
Since $R_{\infty}(\mu,t)$ is periodic with period $4\pi$, then $R_{\infty}(4\pi \mu)$ will be periodic with period $1$, thus if we apply Lemma  \ref{ddgsgs} to function $R_{\infty}(4\pi \mu)$, we have:    $$\sum_{j\in (\Z/k\Z)^m} R_{\infty}(-\frac{4\pi j}{k},t)\frac{|\theta_j|^2_{h_t^k}}{\|\theta_j\|^2_{h_t^k}}\sim k^m(R_{\infty}(\mu,t)+k^{-1}b_{11}(\mu,t)+\cdots),$$
where $\mu=4\pi(y+\nabla \psi_t)$.
In fact, we can apply Lemma \ref{ddgsgs} to each coefficient, e.g.,
 $$\frac{1}{k}\sum_{j\in (\Z/k\Z)^m}A_{1}(-\frac{4\pi j}{k},t)\frac{|\theta_j|^2_{h_t^k}}{\|\theta_j\|^2_{h_t^k}}\sim k^m(k^{-1}A_{1}(\mu,t)+\cdots)$$
and so on, then we have the complete asymptotics:
$$ \begin{array}{lll}\sum_{j\in (\Z/k\Z)^m}R_k(j,t)\frac{|\theta_j|^2_{h_t^k}}{\|\theta_j\|^2_{h_t^k}}&\sim & k^m(R_{\infty}(\mu,t)+k^{-1}(A_{1}+b_{11})+\cdots).
\end{array} $$
We can divide $R_{\infty}$ since in Lemma \ref{jhgf}, we prove this term is strictly positive, uniformly bounded and smooth. Hence,
 $$\begin{array}{lll}&&\frac{1}{k}\log\sum_{j\in (\Z/k\Z)^m}R_k(j,t)\frac{|\theta_j|^2_{h_t^k}}{\|\theta_j\|^2_{h_t^k}} \\ && \\
& \sim & k^{-1}\log [k^mR_{\infty}(\mu,t)(1+\frac{1}{k}\frac{A_{1}+b_{11}}{R_{\infty}}+\cdots)]
 \\ && \\
& \sim &mk^{-1}\log k+k^{-1}\log R_{\infty}+k^{-1}\log (1+\frac{1}{k}\frac{A_{1}+b_{11}}{R_{\infty}}+\cdots)
 \\ && \\
& \sim & mk^{-1}\log k+k^{-1}\log R_{\infty}+k^{-2}\frac{A_{1}+b_{11}}{R_{\infty}}+ \cdots .
\end{array} $$
In the last step, we use the Taylor expansion $\log(1+x)\sim x-\frac{x^2}{2}+\cdots$. Moreover,   $$\frac{1}{k}\log \sum_{j\in (\Z/k\Z)^m}R_k(j,t)\frac{|\theta_j|^2_{h_t^k}}{\|\theta_j\|^2_{h_t^k}}\longrightarrow 0$$ in $C^{\infty}$ topology as $k \rightarrow \infty$. This implies that the Bergman geodesics converge to the geodesic in the \kahler space in $C^{\infty}$ topology.
\end{proof}
\section{ General Lattice }\label{general}

In this section, we will sketch the proof of our main theorem for any principally polarized Abelian variety.

 Let $M=\C^m/\Lambda$ where $\Lambda=$ Span$_{\Z}\{\lambda_1,...,\lambda_{2m}\}$ is a lattice in $\C^m$ with its normalized period matrix given by
$\Omega :=[I,Z]$ where $Z^t=Z$ and $Im Z>0$. 
 Choose $\{x_1,...,x_m,y_1,...,y_m\}$ as the coordinates of the basis dual to  $\{\lambda_1,...,\lambda_{2m}\}$ such that $z_{\alpha}=x_{\alpha}+\sum _{\beta=1}^mZ_{\alpha \beta}y_{\beta}$ and $\bar z_{\alpha}=x_{\alpha}+\sum _{\beta=1}^m\bar Z_{\alpha \beta}y_{\beta}$ \cite{GH}.

 Assume $L\rightarrow M$ is a principal polarization of $M$, then the holomorphic sections of $\Hk$ are given by theta functions (\ref{a}).

 Now consider the \kahler potential in the form:
\begin{equation}\varphi(t,y)=2\pi y  X y^T+4\pi \psi(t,X y^T)\end{equation}
where $y=(y_1,...,y_m)$, $X=Im Z$.  We assume $\varphi$ is convex in $y$ and $\psi$ is smooth on $\R^m$ and periodic with period $1$ in each variable $y_j$ for any fixed $t$.
 Then it's easy to check that such \kahler potential satisfies conditions (\ref{iuyth}).

By choosing such \kahler potential, Proposition \ref{dfghhg} still holds depending on the following computations (see also \cite{FMN}):
 $$\langle\theta_{l'}(z,\Omega ),\theta_{l}(z,\Omega )\rangle_{h^k_t}= \int_{[0,1]^m\times [0,1]^m} \left(\sum_{n'\in \Z^m} e^{-i \pi(l+kn')\cdot\frac{Z}{k}(l+kn')^T}e^{-2\pi i(l+kn')\cdot z}\right)\cdot$$ $$\left(\sum_{n\in \Z^m} e^{-i \pi(l+kn)\cdot\frac{\bar Z}{k}(l+kn)}e^{-2\pi i(l+kn)^T\cdot\bar z}\right)\cdot e^{-2k\pi y X y^T-4k\pi \psi(t,X y)}\cdot \det \hess \varphi_t(y) dxdy$$
 $$=\delta_{l,l'}\sum_{n\in \Z^m} \int_{[0,1]^m} e^{-2k\pi(y+\frac{l+kn}{k})\cdot X(y+\frac{l+kn}{k})^T} e^{-4k\pi \psi(t,X y)}\cdot \det \hess \varphi_t(y) dy.$$ Thus  $\left\{\theta_l(z,\Omega ), \,\, l\in  (\Z/k\Z)^m \right\}$ forms an orthogonal basis of $\Hk$. Furthermore,
\begin{equation}\label{eqati}\|\theta_l(z,\Omega )\|^2_{h^k_t}=\int _{\R^m}e^{-2k\pi(y+\frac{l}{k})X(y+\frac{l}{k})^T}e^{-4k\pi \psi(t,X y)}\cdot  \det \hess \varphi_t(y)dy.\end{equation}
Then all main steps in the model case can be extended to the general case immediately:
\begin{itemize}
\item
Define $u(t,y)$ as the Legendre transform of $\varphi(t,y)$ with respect to $y$ variables for any fixed $t$, then we can still linearize $u(t,y)$ along the geodesics since the Proposition \ref{jghnbv} is only the property of convex functions (p.106 in \cite{R}).
 \item
By substituting $\varphi$ by the Legendre transform $u(t,y)$, we rewrite (\ref{eqati}) as
$$e^{-2k\pi (\frac{j}{k})X(\frac{j}{k})^T}\int_{\R^m} e^{-4k\pi (\nabla u \cdot X \cdot (\frac{j}{k})^T+u-\mu \cdot \nabla u)}d\mu, $$
where $\mu=\nabla \varphi$.

By applying the stationary phase method, we can get the complete asymptotics of this integration evalued at $\mu'=-X\cdot (\frac{4\pi j}{k})^T$ which is the critical point of the phase function. Thus $R_k(j,t)$ which is the ratio of the norming constants will be
 asymptotic to $R_\infty(\mu,t)$ as $$R_k(j,t)\sim R_\infty(\mu,t)(1+k^{-1}a_1(\mu,t)+\cdots )|_{\mu=-X\cdot (\frac{4\pi j}{k})^T}$$ If we change variable as $ \mu\cdot (4\pi X)^{-1} =\nu$, then $R_\infty(\nu,t)$ and each $a_j(\nu,t)$ are smooth functions over $\R^m$ and periodic with period $1$ in variables $\nu$ for any fixed $t$.


\item
In the general case, we define the operator $U: f(x)\rightarrow f(x-\frac{1}{k})$. Then for general theta functions $\theta_l(z,\Omega )$, we still have: $$U(\theta_l(z,\Omega ))=e^{-2\pi i \frac{l}{k}} \theta_l(z,\Omega ),$$ where $e^{-2\pi i \frac{l}{k}}$ denotes $e^{-2\pi i \frac{l_1}{k}}\cdots e^{-2\pi i \frac{l_m}{k}}$.  Then by applying the Weyl quantization to the Bergman kernel and using Fourier transform and Taylor expansion, for any $f(4\pi X \cdot x^T)\in C^\infty(\mathbb R^m)$ which is also periodic with period $1$ in $x$ variables, following the proof in section \ref{dddsgg}, we can prove
$$\frac{1}{k^m}\sum_{j\in (\Z/k\Z)^m}f(- X\cdot(\frac{4\pi j}{k})^T)\frac{ |\theta_j(z,\Omega )|_{h^k}}{\| \theta_j(z,\Omega )\|_{h^k}}\sim f(\mu)+k^{-1}b_1(\mu)+\cdots ,$$ where $\mu= \nabla \varphi$ . Our main result  with the same formula as the model case holds if we apply this formula to $R_\infty(\mu,t)$ and each $a_j(\mu,t)$ and follow the steps in section \ref{dsss}.
\end{itemize}
Thus our main result holds for any  principally polarized Abelian variety.

\section{Complete asymptotics of harmonic maps}\label{testharmonic}
The proof of Theorem \ref{harmin} is similar to the one in \cite{RZ}. For brevity, we just sketch the main steps for the model case $M=\C^m/\Lambda$,
 where $\Lambda=\Z^m+i\Z^m$.

The crucial formula in the toric case is the identity (4.1) in \cite{RZ}, while in our Abelian case, we modify it to be
\begin{equation}\label{comp}\begin{array}{lll}&& \varphi_k(q,z)-\varphi(q,z) \\ && \\
&= &\frac{1}{k} \log 	\sum_{j\in(\Z/k\Z)^m} \exp \left(\int_{\partial N}\partial_{\nu(p)} G(p,q)\log  \|\theta_j(z)\|^2_{h_\psi^k(p)}dV_{\partial N}(p)\right) |\theta_j(z)|^2_{h_\varphi^k(q)},\end{array}\end{equation}
where $G(q,p)$ denotes the positive Dirichlet Green kernel
for the Laplacian $\triangle_ {N,g}$, $dV_{\partial N}$ is the
induced measure on $\partial N$ by restricting the Riemannian volume form
$dV_{N} $ from $N$ to $\partial N$ and $\nu(q)$ is an outward unit normal to
$\partial N$.
Then to prove Theorem \ref{harmin} is equivalent to prove that (\ref{comp}) admits complete asymptotics.
Denote $$K_k(q,j)=\exp \left(-\int_{\partial N}\partial_{\nu(p)} G(p,q)\log  \frac{ \|\theta_j(z)\|^2_{h_\varphi^k(q)}}{\|\theta_j(z)\|^2_{h_\psi^k(p)}}dV_{\partial N}(p)\right)$$
Then we can rewrite (\ref{comp}) as \begin{equation}\label{codrmp}\varphi_k(q,z)-\varphi(q,z) =\frac{1}{k} \log 	 \sum_{j\in(\Z/k\Z)^m} K_k(q,j) \frac{|\theta_j(z)|^2_{h_\varphi^k(q)}}{\|\theta_j(z)\|^2_{h_\varphi^k(q)}}\end{equation}
Put $u_q:=u_{\varphi(q)}=u(q,\cdot)$ is the Legendre transform of $\varphi_q(y) \in \hcal_0^\Gamma$ for $q\in N$. Denote \begin{equation}\label{ddgssd}K_\infty(q,x)=\exp \left(-\frac{1}{2}\int_{\partial N}\partial_{\nu(p)} G(p,q)\log  \frac{\det \hess u_q(x)}{\det \hess u_p(x)}dV_{\partial N}(p)\right)\end{equation} where $x=\nabla \varphi$.

From the proof of the Regularity Lemma \ref{jhgf}, if we plug in the complete asymptotic expansion of the norming constants $\|\theta_j(z)\|^2_{h_\varphi^k(q)}$ and $\|\theta_j(z)\|^2_{h_\psi^k(p)}$, we have the following complete asymptotic expansion,
\begin{equation}\label{dfddgssd}K_k(q,j)=K_\infty(q,x)+k^{-1}b_1(q,x)+\cdots |_{x=-\frac{4\pi j}{k}}\end{equation}

If we plug  (\ref{dfddgssd}) into the right hand side of (\ref{codrmp}), we obtain the following expansion,
$$\begin{array}{lll}\frac{1}{k}\log \left(\sum_j K_\infty(q,\frac{-4\pi j}{k}) \frac{|\theta_j(z)|^2_{h_\varphi^k(q)}}{\|\theta_j(z)\|^2_{h_\varphi^k(q)}}+k^{-1}\sum_j b_1(q,\frac{-4\pi j}{k}) \frac{|\theta_j(z)|^2_{h_\varphi^k(q)}}{\|\theta_j(z)\|^2_{h_\varphi^k(q)}}+\cdots \right).\end{array}$$

Hence, Theorem \ref{harmin} follows if we apply the generalized Bernstein Lemma \ref{ddgsgs} to each summation above and follow the steps in section \ref{dsss}.

 \end{document}